\documentclass[a4paper,11pt]{article}
\usepackage[dvipdfmx]{graphicx}
\usepackage{amsmath,amsthm,amssymb}
\usepackage{enumerate}
\usepackage{graphicx,here}
\usepackage{tikz}
\usepackage{titlesec}
\usepackage{empheq}
\usepackage{color}
\usepackage{xfrac,bigints}
\usepackage[hidelinks]{hyperref}

\usepackage{subcaption}

\newtheorem{thm}{Theorem}[section]
\newtheorem{lem}[thm]{Lemma}
\newtheorem{prop}[thm]{Proposition}
\newtheorem{cor}[thm]{Corollary}
\newtheorem{dfn}[thm]{Definition}
\newtheorem{rem}[thm]{Remark}

\newtheorem{ex}[thm]{Example}

\titleformat*{\section}{\normalsize\bfseries}
\titleformat*{\subsection}{\normalsize\bfseries}

\def\N{{\mathbb N}}
\def\R{{\mathbb R}}
\def\Rp{{\mathbb R}_+}
\def\Rb{\overline{{\mathbb R}}_+}

\def\DS{\displaystyle}
\def\supp{\mbox{\rm supp}}
\def\vep{\varepsilon}

\begin{document}

\begin{flushright}
\footnotesize{\today}
\end{flushright}
\begin{center}
{\large\bf
Mass conservation and gelation\\ for the Smoluchowski coagulation equation:\\ a generalized moment approach
}
\end{center}
\begin{center}
  Masato Kimura$^{\mbox{\scriptsize a}}$,
  Hisanori Miyata$^{\mbox{\scriptsize b}}$,
\begin{tabular}{l}
  {\scriptsize a)} Faculty of Mathematics and Physics, Kanazawa University\\
{\scriptsize b)} Graduate School of Natural Science and Technology, Kanazawa University
\end{tabular}
\end{center}

\renewcommand{\thefootnote}{\fnsymbol{footnote}}
\footnote[0]{E-mail: mkimura@se.kanazawa-u.ac.jp (MK)}
\renewcommand{\thefootnote}{\arabic{footnote}}

\vspace{0.5cm}
\begin{abstract}
The Smoluchowski coagulation equation (SCE) is a population balance model that describes the time evolution of cluster size distributions resulting from particle aggregation. Although it is formally a mass-conserving system, solutions may exhibit a gelation phenomenon—a sudden loss of mass—when the coagulation kernel grows superlinearly. In this paper, we rigorously analyze mass conservation and gelation for weak solutions to the SCE with inhomogeneous coagulation kernels. By introducing a generalized moment framework, we derive sharp sufficient conditions for both mass conservation and gelation, expressed in terms of the initial data and the properties of the coagulation kernel.
\end{abstract}

\tableofcontents

\section{Introduction}\label{sec:1}
\setcounter{equation}{0}

The phenomenon of coagulation has been observed in various contexts worldwide and has been studied across multiple disciplines, including biology, chemical engineering, and physics. Among the various models describing such phenomena, the Smoluchowski coagulation equation (SCE) plays a fundamental role in modeling the coagulation of particles.

Originally proposed by Marian Smoluchowski in 1916 as a discrete model~\cite{Smoluchowski1916}, and later extended by M\"{u}ller to a continuous setting~\cite{Muller1928}, the following continuous form of the SCE is now widely studied:
\begin{equation}\label{SCE}
  \partial_t u(x,t)=\frac{1}{2} \int_{0}^{x}K(x-y,y)u(y,t)u(x-y,t)\,dy
  - \int_{0}^{\infty}K(x,y)u(x,t)u(y,t)\,dy,
\end{equation}
where $x \in \Rb \coloneqq [0, \infty)$ and $t \in I \coloneqq [0, T]$ for some $T > 0$.

The first integral term on the right-hand side of \eqref{SCE} represents the rate at which clusters of size $x$ are formed due to coagulation, where two clusters of size $y$ and $x - y$ combine at a rate given by $K(x - y, y)$ for $y \in [0, x]$. The factor $1/2$ is included to avoid double-counting. In contrast, the second integral term describes the rate at which clusters of size $x$ are lost due to coagulation with other clusters, resulting in larger clusters of size $x + y$.

The coagulation model \eqref{SCE} finds applications in diverse areas, including wastewater treatment~\cite{Z-O-R-V2015}, astrophysics~\cite{Estrada-Cuzzi2008,Kolesnichenko2020}, aerosol science~\cite{Ferreira2021}, protein aggregation~\cite{Z-K-R2018}, and nanoparticle clustering~\cite{Alexandrov2021}. In this paper, we refer to \eqref{SCE} simply as the Smoluchowski coagulation equation (SCE).

We consider the initial value problem for \eqref{SCE} with an initial condition $u_0 \in C^0(\Rb; \Rb)$ that satisfies the finite mass condition: $\int_{0}^{\infty} x u_0(x)\,dx < \infty$. We also assume that the coagulation kernel $K(x, y)$ satisfies the following properties:
\begin{align}\label{Kcond}
K \in L^1_{\rm loc} (\Rb \times \Rb), \quad K(x, y) = K(y, x) \ge 0,
\end{align}
where $x \in \Rb$ denotes the cluster size (volume or mass). The function $u(x,t)$ describes the size distribution of clusters at time $t \in I$, and $K(x, y)$ governs the rate at which clusters of sizes $x$ and $y$ coagulate. The functional form of $K$ varies depending on the physical context~\cite{Aldous1999, S-H-P1994}.

It is known that the SCE \eqref{SCE} can be formally written as the following conservation law~\cite{Filbet-Laurencot2004, M-F-F-K1998, T-I-N1996}:
\begin{align}\label{conservation law}
  x \partial_t u(x,t) = -\partial_x J[u](x,t), \quad (x,t) \in \Rb \times I,
\end{align}
where the mass flux $J[u](x,t)$ denotes the rate of mass transfer across the point $x$ in the size domain $\Rb$, from the region $[0, x)$ to $(x, \infty)$. A detailed derivation is provided in the appendix of~\cite{I-K-M2023}.

Letting $M_1(t) \coloneqq \int_0^\infty x u(x,t)\,dx$ denote the total mass, we formally obtain:
\begin{align*}
M_1(t) - M_1(0)
&= \int_0^\infty \int_0^t \partial_s(xu(x,s))\,ds\,dx\\
&= -\int_0^t \int_0^\infty \partial_x J[u](x,s)\,dx\,ds\\
&= -\int_0^t (J[u](\infty,s) - J[u](0,s))\,ds\\
&= -\int_0^t J[u](\infty,s)\,ds,
\end{align*}
since $J[u](0,s) = 0$ by definition. We define $I_\infty(t) \coloneqq \int_0^t J[u](\infty,s)\,ds \ge 0$. If $I_\infty(t) = 0$, mass is conserved, i.e., $M_1(t) = M_1(0)$. If $I_\infty(t) > 0$, then $M_1(t) < M_1(0)$, indicating the occurrence of gelation. Although the above is only a formal argument, it is rigorously justified in Section~\ref{sec:5} (see Lemma~\ref{Sahidul-Kimura-Lem2} and Theorem~\ref{mainth1}).

For instance, consider the homogeneous coagulation kernel $K(x, y) = x^\alpha y^\beta + x^\beta y^\alpha$ with $\alpha, \beta \in [0, 1]$. It is known that weak solutions fail to conserve mass if $\alpha + \beta > 1$, whereas mass is conserved when $\alpha + \beta \le 1$ (see~\cite{E-L-M-P2003}).  
As examples of mass conservation and gelation, Figures~\ref{fig:1} and \ref{fig:2} show numerical results for $K(x, y) = (xy)^{1/2}$ and $K(x, y) = xy + x + y$, respectively, using the finite volume method from~\cite{Filbet-Laurencot2004}. The initial condition in both cases is $u_0(x) = e^{-x}$. Figures~\ref{fig:1a} and \ref{fig:2a} show the time evolution of $u(x,t)$ for $t \in [0,3]$, with color corresponding to time. Figures~\ref{fig:1b} and \ref{fig:2b} display the evolution of total mass $M_1(t)$. For $K(x, y) = (xy)^{1/2}$, where the kernel has at most linear growth, mass is conserved. However, for $K(x, y) = xy + x + y$, which includes quadratic growth, a decrease in total mass over time is observed, indicating gelation.

\begin{figure}[h]
  \centering
  \begin{minipage}{0.43\columnwidth}
    \centering
    \includegraphics[width=\columnwidth]{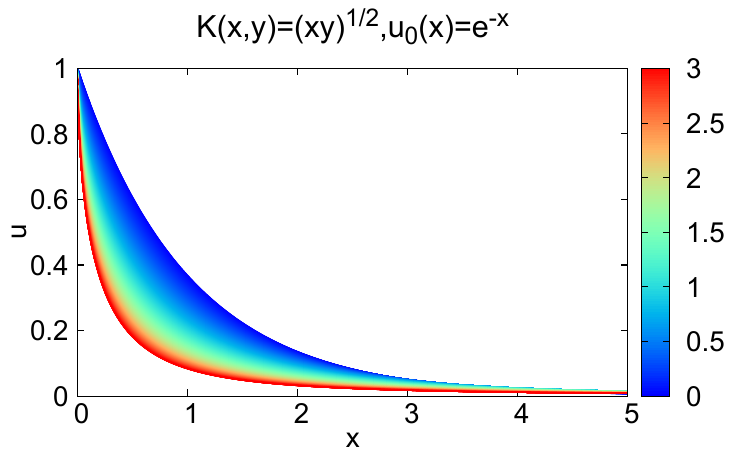}
    \subcaption{$u(x,t)$ $(x\in [0,5],~t\in [0,3])$}
    \label{fig:1a}
  \end{minipage}
  \hspace{5mm}
  \begin{minipage}{0.43\columnwidth}
    \centering
    \includegraphics[width=\columnwidth]{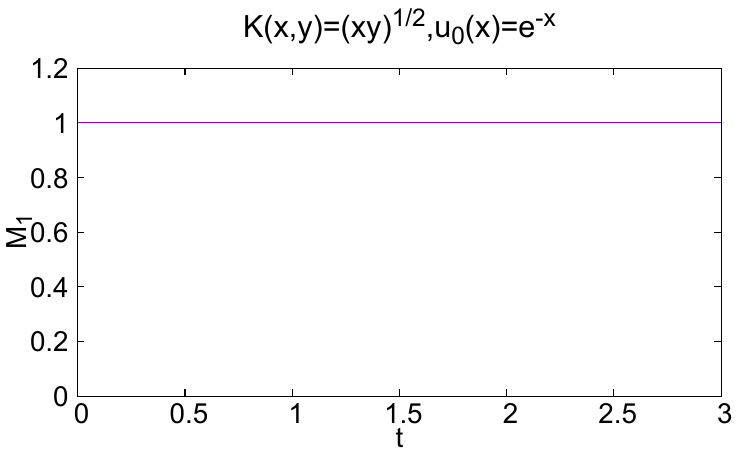}
    \subcaption{$M_1(t)$ $(t\in [0,3])$}
    \label{fig:1b}
  \end{minipage}
  \caption{Numerical result for $K(x,y)=(xy)^{1/2}$ and $u_0(x)=e^{-x}$}
  \label{fig:1}
\end{figure}
\begin{figure}[h]
  \centering
  \begin{minipage}{0.43\columnwidth}
    \centering
    \includegraphics[width=\columnwidth]{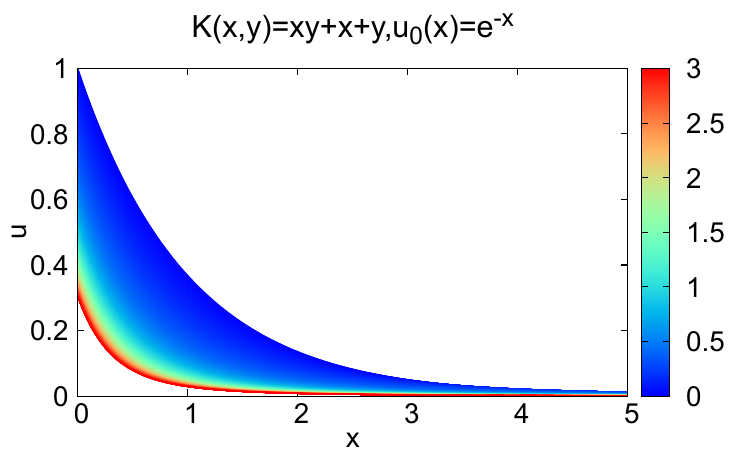}
    \subcaption{$u(x,t)$ $(x\in [0,5],~t\in [0,3])$}
    \label{fig:2a}
  \end{minipage}
  \hspace{5mm}
  \begin{minipage}{0.43\columnwidth}
    \centering
    \includegraphics[width=\columnwidth]{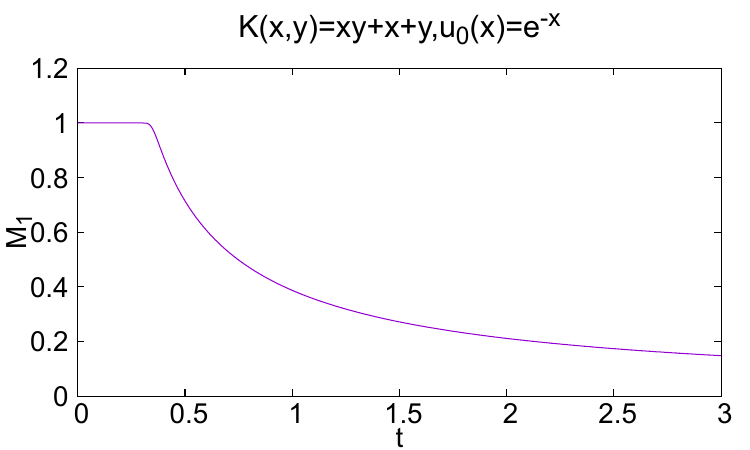}
    \subcaption{$M_1(t)$ $(t\in [0,3])$}
    \label{fig:2b}
  \end{minipage}
  \caption{Numerical result for $K(x,y)=xy+x+y$ and $u_0(x)=e^{-x}$}
  \label{fig:2}
\end{figure}

The existence and uniqueness of solutions to the SCE, as well as qualitative properties such as gelation, have been the subject of extensive mathematical research. A comprehensive review is beyond the scope of this paper and the expertise of the authors. For foundational results on existence and other mathematical properties, we refer the reader to~\cite{E-L-M-P2003,E-M-P2002,Friedman1998,McLeod1962,Melzak1957}; for recent generalizations and developments, see~\cite{A-L-S2021,G-L-S2025,Xie2023} and the references therein.

The objective of this paper is to refine the generalized moment method introduced in~\cite{I-K-M2023} for analyzing mass conservation and gelation in the SCE. We rigorously develop the notion of weak solutions using generalized moments and investigate the behavior of generalized moments associated with these solutions. We also derive sufficient conditions under which weak solutions exhibit mass conservation or gelation.

The structure of this paper is as follows. Section~\ref{sec:2} introduces generalized and truncated generalized moments, along with a proposition concerning sub/super-additivity and convexity/concavity of the weight function. In Section~\ref{sec:3}, we define weak solutions via generalized moments and compare them with the definition in~\cite{E-L-M-P2003}. The mass flux $J[u]$ is introduced, and it is shown that the conservation law \eqref{conservation law} holds in a weak sense (see Proposition~\ref{weaksolprop}). Section~\ref{sec:4} explores the precise properties of generalized and truncated moments for weak solutions. In Section~\ref{sec:5}, we present sufficient conditions on $K(x,y)$ and $u_0(x)$ to ensure mass conservation, extending the results of~\cite{I-K-M2023} from classical to weak solutions. Section~\ref{sec:6} provides sufficient conditions for the occurrence of gelation, again using the generalized moment approach. Finally, Section~\ref{sec:7} summarizes the main results and outlines possible future directions.

This paper is based in part on the master's thesis of one of the authors~\cite{Miyata2025}. Throughout, we use the notation $\Rp := (0, \infty)$ and $\Rb := [0, \infty)$.

\section{Generalized moment framework}\label{sec:2}
\setcounter{equation}{0}

We begin by introducing the $k$-th order moment and truncated moment of a solution $u(x,t)$.

\begin{dfn}[Moment and truncated moment of $k$-th order]\label{moment}
  For $T>0$, we set $I\coloneqq[0,T]$.
  Let $u:\Rb\times I\to \Rb$ be a nonnegative measurable function such that
  $u(\cdot,t)\in L^1_{\rm loc}(\Rb)$ for $t\in I$.
  For $k\in \N\cup \{0\}$, we define the $k$-th order truncated moment $m_k(r,t)$ and moment
  $M_k(t)$ as follows:
  \begin{align*}
    m_k(r,t)&\coloneqq \int^r_0 x^k u(x,t)\,dx \in [0,\infty) \quad(r\geq 0,\ t\in I),\\
    M_k(t)&\coloneqq \lim_{r\to \infty} m_k(r,t) = \int^\infty_0 x^k u(x,t)\,dx \in [0,\infty] \quad(t\in I).
  \end{align*}
\end{dfn}

\begin{rem}
{\rm
  The zeroth-order moment $M_0(t)$ and the first-order moment $M_1(t)$ represent the ``total number of clusters'' and the ``total mass,'' respectively.
}
\end{rem}

The primary analytical tool of this paper is the generalized moment with an arbitrary weight function, as developed in \cite{I-K-M2023}.

\begin{dfn}[Generalized moment and truncated generalized moment]\label{gmoment}
Under the same assumptions as in Definition~\ref{moment},
for a weight function $b\in C^0(\Rb)$, we define the truncated generalized moment $m^b(r,t)$ and generalized moment $M^b(t)$ as follows:
\begin{align}
m^b(r,t)&\coloneqq \int^r_0 b(x) u(x,t)\,dx \quad(r\geq 0,\ t\in I),\label{mbr}\\
M^b(t)&\coloneqq  \lim_{r\to \infty} m^b(r,t)=\int^\infty_0 b(x) u(x,t)\,dx
\quad(t\in I)\quad\text{(if it exists)}.\notag
\end{align}
In particular, if $b\in C^0(\Rb;\R)$, then $M^b(t)$ exists and $M^b(t)\in [0,\infty]$ for all $t\in I$.
\end{dfn}

In many instances, the weight function $b(x)$ is assumed to possess subadditivity or superadditivity. Their definitions and useful properties are given below.

\begin{dfn}[Subadditivity and superadditivity]
A real-valued function $b$ defined on $\Rb$ is called \emph{subadditive} if
$b(x+y)\leq b(x)+b(y)$ holds for $x,y\in\Rb$.
Similarly, $b$ is called \emph{superadditive} if $b(x+y)\geq b(x)+b(y)$ holds for $x,y\in\Rb$.
\end{dfn}

\begin{prop}\label{additive}
Let $b$ be a function defined on $\Rb$.
\begin{enumerate}[(1)]
\item
If $b$ is concave on $\Rb$, then
$b(0)+b(x+y)\le b(x)+b(y)$ holds for $x,y\in\Rb$.
Moreover, if $b(0)\ge 0$, then $b$ is subadditive on $\Rb$.
\item
If $b$ is strictly concave on $\Rb$, then
$b(0)+b(x+y)< b(x)+b(y)$ holds for $x,y\in\Rp$.
Moreover, if $b(0)\ge 0$, then $b$ is strictly subadditive on $\Rp$.
\item
If $b$ is convex on $\Rb$, then
$b(0)+b(x+y)\ge b(x)+b(y)$ for $x,y\in\Rb$. Moreover, if $b(0)\le 0$, then $b$ is superadditive on $\Rb$.
\item
If $b$ is strictly convex on $\Rb$, then
$b(0)+b(x+y)> b(x)+b(y)$ for $x,y\in\Rp$.
Moreover, if $b(0)\le 0$, then $b$ is strictly superadditive on $\Rp$.
\end{enumerate}
\end{prop}

\begin{proof}
Assume $b$ is concave on $\Rb$. Then for any $z\in\Rp$ and $\theta\in (0,1)$, it holds that
\begin{align}\label{thetaz}
b(\theta z)=b(\theta z+(1-\theta) 0)\ge \theta b(z)+(1-\theta) b(0).
\end{align}
Since $b(0)+b(x+y)\le b(x)+b(y)$ holds if $x=0$ or $y=0$, we suppose
$x>0$ and $y>0$.
For $x,y>0$, set $z = x+y$ and $\theta = x/(x+y)$, $y/(x+y)$ to obtain
\begin{align*}
  b(x)+b(y)&=b\left(\frac{x}{x+y}(x+y)\right)+b\left(\frac{y}{x+y}(x+y)\right)\\
  &\geq \left\{\frac{x}{x+y}b(x+y)+\frac{y}{x+y}b(0)\right\}
  +\left\{\frac{y}{x+y}b(x+y)+\frac{x}{x+y}b(0)\right\}\\
    &=b(x+y)+b(0).
  \end{align*}
Thus, $b(x)+b(y)\ge b(x+y)+b(0)$ holds, and if $b(0)\ge 0$, then $b$ is subadditive.
For the case that $b$ is strictly concave, the inequality sign in \eqref{thetaz} is replaced by $>$
and the strict subadditivity similarly follows.

When $b$ is a convex function on $\Rb$, since $-b$ is concave on $\Rb$,
$-b(0)-b(x+y)\le -b(x)-b(y)$ holds for $x,y\in\Rb$. This implies the third and fourth statements.
\end{proof}

\section{Classical and weak solutions}\label{sec:3}
\setcounter{equation}{0}

Hereafter, we assume that the coagulation kernel $K(x,y)$ satisfies the conditions given in \eqref{Kcond}. For brevity, these assumptions will not be repeated, but are understood throughout the paper.

For convenience, we rewrite the SCE \eqref{SCE} as follows:
\[
\partial_t u(x,t) = Q[u](x,t),
\]
where we define:
\begin{align*}
  Q_1[u](x,t) &\coloneqq \frac{1}{2}\int_{0}^{x}K(x-y,y)u(y,t)u(x-y,t)\,dy,\\
  Q_2[u](x,t) &\coloneqq \int_{0}^{\infty}K(x,y)u(x,t)u(y,t)\,dy,\\
  Q[u](x,t) &\coloneqq Q_1[u](x,t)-Q_2[u](x,t).
\end{align*}
Here, $Q_1[u]$ and $Q_2[u]$ respectively represent the rates of increase and decrease of the density of clusters of size $x$.

We now introduce the notions of classical and weak solutions to the initial value problem for the SCE and examine their properties.

\begin{dfn}[Classical solution] \label{classicalsolution}
Let $T > 0$ and define $I \coloneqq [0,T]$. Suppose $u_0 \in C^0(\Rb; \Rb)$ and the coagulation kernel $K$ satisfies \eqref{Kcond} and $K \in C^0(\Rb \times \Rb)$. A function $u: \Rb \times I \to \Rb$ is called a classical solution to the SCE \eqref{SCE} on $I$ with initial condition $u(x,0) = u_0(x)$ if:
  \begin{enumerate}[{\rm (C1)}]
    \item $u\in C^0(\Rb \times I;\Rb)$ and $\partial_t u\in C^0(\Rb \times I)$;
    \item $u(x, 0) = u_0(x)\;(x\in \Rb )$;
    \item
      $\DS \left[(x,t)\mapsto \int^\infty_0 K(x,y)u(y,t)\,dy\right]\in C^0(\Rb \times I)$;
    \item $\displaystyle\partial_t u(x,t)=Q[u](x,t)\;((x,t)\in\Rb \times I)$.
  \end{enumerate}
If $u$ satisfies these conditions on $[0,T]$ for every $T > 0$, then we say that $u$ is a classical solution on $\Rb$.
\end{dfn}

For classical solutions $u$, the condition $M_1(0) < \infty$ implies $M_0(0) < \infty$, and both $M_0(t)$ and $M_1(t)$ are known to be non-increasing in time (see Proposition~2 in \cite{I-K-M2023} or Proposition~\ref{Monotonically decreasing property of the moment}).

We define the weighted Lebesgue space:
\[
L^1_1(\Rp) \coloneqq \left\{ f \in L^1(\Rp) \,\middle|\, \int_0^\infty (1+x)|f(x)|\,dx < \infty \right\},
\]
which is a Banach space under the norm
\[
\|f\|_{L^1_1(\Rp)} \coloneqq \int_0^\infty (1+x)|f(x)|\,dx.
\]
A nonnegative function $u(\cdot,t)$ belongs to $L^1_1(\Rp)$ if and only if both $M_0(t)$ and $M_1(t)$ are finite. Thus, for classical solutions, if $u(\cdot,0) \in L^1_1(\Rp)$, then $u(\cdot,t) \in L^1_1(\Rp)$ for all $t \in I$.

Next, we introduce the notion of weak solutions.

\begin{dfn}[Weak solution] \label{weaksol}
Let $T > 0$ and $I \coloneqq [0,T]$. Suppose $u_0 \in L^1_1(\Rp)$ and $u_0(x) \ge 0$. A function $u$ is called a weak solution to the initial value problem for the SCE \eqref{SCE} on $I$ with initial data $u(x,0) = u_0(x)$ if:
\begin{enumerate}[{\rm (W1)}]
  \item $u \in L^\infty(0,T; L^1_1(\Rp))$;
  \item $u(x,t) \ge 0$ for almost every $(x,t) \in \Rp \times I$;
  \item $Q_1[u], Q_2[u] \in L^1_{\rm loc}(\Rb \times I)$;
  \item For all $b \in C^0(\Rb)$, $r > 0$, and $t \in I$,
  \[
  m^b(r,t) = \int_0^r b(x)u_0(x)\,dx + \int_0^t \int_0^r b(x)Q[u](x,s)\,dxds,
  \]
  where $m^b(r,t)$ is the truncated generalized moment defined in \eqref{mbr}.
\end{enumerate}
If $u$ satisfies these conditions for all $T > 0$, we say that $u$ is a weak solution on $\Rb$.
\end{dfn}

Hereafter, we assume that the initial data $u_0$ belongs to $L^1_1(\Rp)$; this assumption will not be repeated explicitly.

We state a property of the generalized moment for weak solutions. It follows directly from (W3) and (W4), so we omit the proof.

\begin{prop}\label{weaksolpr}
Let $u$ be a weak solution to the SCE on $I = [0,T]$, and let $m^b$ be the truncated generalized moment with weight $b \in C^0(\Rb)$. Then, for each $r \in \Rb$, $m^b(r,\cdot) \in W^{1,1}(0,T) \subset C^0(I)$, and 
\[
\partial_t m^b(r,t) = \int_0^r b(x)Q[u](x,t)\,dx \quad (\text{a.e. } t \in I), \quad
m^b(r,0) = \int_0^r b(x)u_0(x)\,dx.
\]
\end{prop}

\begin{prop}
Suppose $u_0 \in C^0(\Rb) \cap L^1_1(\Rp)$ and $K$ satisfies \eqref{Kcond} with $K \in C^0(\Rb \times \Rb)$. Let $I = [0,T]$. Then a function $u$ is a classical solution to the SCE on $I$ if and only if it is also a weak solution on $I$.
\end{prop}

\begin{proof}
Let $u$ be a classical solution on $I$.
Since $M_0(t)\le M_0(0)$ and $M_1(t)\le M_1(0)$ hold for $t\in I$ from Proposition~2 of \cite{I-K-M2023} (or Proposition~\ref{Monotonically decreasing property of the moment} in this paper),
we have
\begin{align*}
  \|u(t)\|_{L^1_1(\R_+)}=M_0(t)+M_1(t)\le M_0(0)+M_1(0)= \|u_0\|_{L^1_1(\R_+)}\quad (t\in I),
\end{align*}
and (W1) holds with $\|u\|_{L^\infty(0,T;L^1_1(\R_+))}\le \|u_0\|_{L^1_1(\R_+)}$.
The first condition of (C1) implies (W2). From \eqref{Kcond} and (C1), since $K$ and $u$ are continuous,
$Q_1[u]\in C^0(\R_+\times I)$ holds. From (C3), $Q_2[u]\in C^0(\R_+\times I)$ holds, too.
Therefore, $u$ satisfies (W3). The condition (W4) follows from (C2) and (C4) as
\begin{align*}
  m^b(r,t) &= m^b(r,0)+\int_{0}^{r}b(x)\big(u(x,t)-u(x,0)\big)\,dx\\
  &= m^b(r,0)+\int_{0}^{r}\int_{0}^{t}b(x)\partial_s u(x,s)\,dsdx\\
  &= \int_{0}^{r}b(x)u_0(x)\,dx+\int_{0}^{t}\int_{0}^{r}b(x)Q[u](x,s)\,dxds.
\end{align*}

Conversely, if $u$ is a weak solution satisfying (C1) and (C3), then by Proposition~\ref{weaksolpr}, the remaining classical conditions (C2) and (C4) are satisfied.
\end{proof}

Escobedo et al.~\cite{E-L-M-P2003} provide an alternative definition of weak solution. A function $u$ is a weak solution on $\Rb$ in their sense if it satisfies:
\begin{enumerate}[(E1)]
\item $u\in L^\infty(0,T;L^1_1(\Rp))$ $(T>0);$\label{E1}
\item $u(x,t)\geq 0\quad (\text{a.e.}~ (x,t)\in \Rb\times \Rb)$;\label{E2}
\item $Q_1[u],Q_2[u]\in L^1_{\rm loc}(\Rb\times \Rb)$;\label{E3}
\item \label{E4}
$\displaystyle\int_{0}^{\infty}\varphi(x)u(x,t)\,dx=\int_{0}^{\infty}\varphi(x)u_0(x)\,dx+\int_{0}^{t}\int_{0}^{\infty}\varphi(x)Q[u](x,s)\,dxds$\\
$(\varphi\in C_0^\infty(\Rb),~t\in\Rb)$;
\item $u\in C^0 (\Rb;\mbox{w-}L^1(\Rp))$ and $u(\cdot,0)=u_0$ in $L^1(\Rp)$.
\end{enumerate}
Here, $\mathrm{w\text{-}}L^1(\Rp)$ denotes the weak topology on $L^1(\Rp)$, meaning:
\[
u \in C^0(\Rb; \mathrm{w\text{-}}L^1(\Rp)) \iff \left[t \mapsto \int_0^\infty \varphi(x)u(x,t)\,dx\right] \in C^0(\Rb) \quad (\forall \varphi \in L^\infty(\Rp)).
\]
\begin{prop}\label{wsprop}
For a given $u_0\in L^1_1(\Rp)$,
$u$ is a weak solution to SCE on $\Rb$ if and only if $u$ satisfies {\rm (E1-4)}.
\end{prop}

\begin{proof}
Since the conditions (W1), (W2), and (W3) for any $T>0$ are
equivalent to (E1), (E2), and (E3), respectively,
it is enough to show that (W1–4) implies (E4) and that
(E1–4) implies (W4).

We suppose that the conditions (W1–4) hold for any $T>0$.
Then, for any $t\in\Rb$ and $\varphi\in C^\infty_0(\Rb)$,
choosing $r>0$ such that $\supp(\varphi)\subset [0,r]$,
(E4) follows from (W4).

Conversely, suppose that the conditions (E1–4) hold.
Then, for any $b\in C^0(\Rb)$ and $r>0$,
define
\begin{align*}
b_{r,\varepsilon}(x)\coloneqq\eta_\varepsilon * b_r(x),\quad
b_r(x)\coloneqq
\begin{cases}
  b(x) & (0\leq x \leq r),\\
  b(-x) & (-r\leq x < 0), \\
  0 & (|x|>r),
\end{cases}
\end{align*}
where $\eta_\varepsilon$ is the Friedrichs mollifier.
Then, $b_{r,\varepsilon} \in C^\infty_0(\Rb)$ and
\begin{align*}
  \lim_{\varepsilon\to 0}
  \int_0^\infty b_{r,\varepsilon}(x) v(x)\,dx =\int_0^r b(x)v(x)\,dx
  \quad (v\in L^1_{\rm loc}(\Rb)).
\end{align*}

From (E4), for $t\in\Rb$,
\begin{align}\label{EW4}
  \int_{0}^{\infty}b_{r,\varepsilon}(x)u(x,t)\,dx
  =\int_{0}^{\infty}b_{r,\varepsilon}(x)u_0(x)\,dx
  +\int_{0}^{t}\int_{0}^{\infty}b_{r,\varepsilon}(x)Q[u](x,s)\,dxds.
\end{align}
Taking the limit $\varepsilon\to 0$ in \eqref{EW4}, and using the dominated convergence theorem (since $u(\cdot,t)$, $u_0$, and $\int_{0}^{t}Q[u](\cdot,s)\,ds\in L^1(0,\infty)$), we obtain (W4).
\end{proof}

Finally, for weak solutions $u$ of the SCE, we define the \emph{mass flux} $J[u]$ by
\begin{equation} \label{flux}
  J[u](x,t) \coloneqq \int^x_0 \int^\infty _{x-y} yK(y,z)u(y,t)u(z,t) \, dzdy
  \quad((x,t)\in\Rb \times I),
\end{equation}
with the convention that $J[u](0,t) \coloneqq 0$ for $t\in I$.

\begin{prop}\label{weaksolprop}
For weak solutions $u$ to the SCE on $I=[0,T]$, the following properties hold:
\begin{enumerate}[(1)]   
\item 
$J[u](x,t)<\infty
\quad (x\in\Rb,\text{a.e.}\;t\in I)$.
\item 
$\displaystyle J[u](x,t)=-\int_{0}^{x}yQ[u](y,t)\,dy
\quad (x\in\Rb,~\text{a.e.}\;t\in I)$.
\item 
$\displaystyle 
\int_{0}^{t}J[u](x,s)ds 
= \int_0^x y (u_0(y)-u(y,t))\,dy
= m_1(x,0)-m_1(x,t)
\quad (x\in\Rb,~t\in I)$.
\end{enumerate}
\end{prop}
\begin{proof}
(1) From (W3), for a.e. $t\in I$, we estimate
\begin{align*}
J[u](x,t)
&=
\int_{0}^{x}yu(y,t)\left(\int_{x-y}^{\infty}K(y,z)u(z,t)\,dz\right) dy\\
&\le
x\int_{0}^{x}u(y,t)
\left(\int_{0}^{\infty}K(y,z)u(z,t)\,dz\right)dy
=
x\int_{0}^{x}Q_2[u](y,t)\,dy <\infty.
\end{align*}
\noindent (2) 
From the definition of $J[u]$, we calculate
\begin{align}
J[u](x,t)&=\int_{0}^{x}\int_{x-y}^{\infty}yK(y,z)u(y,t)u(z,t)\,dzdy\notag\\
&=\int_{0}^{x}\left(\int_{0}^{\infty}yK(y,z)u(y,t)u(z,t)\,dz-\int_{0}^{x-y}yK(y,z)u(y,t)u(z,t)\,dz\right)dy\notag\\
&=\int_{0}^{x}yQ_2[u](y,t)\,dy-A,\label{JA}
\end{align}
where
\begin{align*}
A&\coloneqq\int_0^x\int_{0}^{x-y}yK(y,z)u(y,t)u(z,t)\,dzdy\\
&=\int_{0}^{x}\int_{y}^{x}yK(y,s-y)u(y,t)u(s-y,t)\,dsdy\quad\mbox{(set $s=z+y$)}\\
&=\int_{0}^{x}\int_{0}^{s}yK(y,s-y)u(y,t)u(s-y,t)\,dyds.\quad\mbox{(Fubini)}
\end{align*}
Changing the integral variable in the last integral from $y$ to $z=s-y$:
\begin{align*}
A&=\int_{0}^{x}\int_{0}^{s}(s-z)K(s-z,z)u(s-z,t)u(z,t)\,dzds\\
&=\int_{0}^{x}\int_{0}^{s}sK(s-z,z)u(s-z,t)u(z,t)\,dzds-\int_{0}^{x}\int_{0}^{s}zK(s-z,z)u(s-z,t)u(z,t)\,dzds\\
&=\int_{0}^{x}2sQ_1[u](s,t)ds-A
\end{align*}
This implies 
$A=\int_{0}^{x}sQ_1[u](s,t)ds$, and substituting back into \eqref{JA}, we obtain
\begin{align*}
    J[u](x,t)&=\int_{0}^{x}\left(yQ_2[u](y,t)-yQ_1[u](y,t)\right)\,dy
    =-\int_{0}^{t}yQ[u](y,t)\,dy.
\end{align*}
\noindent
(3) From Proposition~\ref{weaksolpr} with $b(y)=y$, we have
\begin{align*}
\int_0^x y (u(y,t)-u_0(y))\,dy
&=m_1(x,t)-m_1(x,0)\\
&=
\int_{0}^{t} \left(\int_0^x yQ[u](y,s)\,dy\right) ds\\
&=
-\int_{0}^{t} J[u](x,s)\,ds,
\end{align*}
which yields the result.
\end{proof}

\section{Analytical properties of generalized moments}\label{sec:4}
\setcounter{equation}{0}
This section summarizes detailed properties of the (truncated) generalized moments.
We first introduce a well-known identity for generalized moments with compactly supported weight functions. Then we present a refined version—the truncated generalized moment identity—and provide some of its applications.
Although the generalized moment identity is not directly used in later sections, we include it here with a proof for the reader's convenience.

\begin{prop}[Generalized moment identity \cite{Pego2009}]
Let $u$ be a weak solution to the SCE on $I=[0,T]$.
Then, for any compact-support weight function $a\in C^0_0(\Rb )$, the following identity holds:
\begin{align*}
\frac{d}{dt}M^{a}(t)=\frac{1}{2}\int_{0}^{\infty}\int_{0}^{\infty}
\big(a(x+y)-a(x)-a(y)\big)K(x,y)u(x,t)u(y,t)dydx
\quad (\text{a.e.}~t\in I).
\end{align*}
\end{prop}
\begin{proof}
We choose $r>0$ to satisfy $[0,r)\supset \supp (a)$.
Then, since $M^a(t)=m^a(r,t)$ holds, 
applying Proposition~\ref{weaksolpr}, for a.e. $t\in I$, we obtain
\begin{align*}
\frac{d}{dt}M^a(t)
=
\partial_tm^a(r,t)
=\int^\infty_0 a(x)Q[u](x,t)\, dx 
=I_1-I_2,
\end{align*}
where we set
\begin{align*}
I_j:=\int_{0}^\infty a(x)Q_j[u](x,t)\,dx
\quad (j=1,2).
\end{align*}
The integral $I_1$ can be transformed as follows:
\begin{align*}
I_1&=\frac{1}{2}\int^\infty_0 \int^x_0 a(x)K(x-y,y)u(x-y,t)u(y,t)\,dydx\\
&=\frac{1}{2}\int_0^\infty \int_y^\infty a(x)K(x-y,y)u(x-y,t)u(y,t)\,dxdy\quad\mbox{(Fubini)}\\
&=\frac{1}{2}\int_0^\infty \int_0^\infty a(z+y)K(z,y)u(z,t)u(y,t)\,dzdy\quad\mbox{(change variable from $x$ to $z=x-y$)}
\end{align*}

On the other hand, by the symmetry of $K$, 
\begin{align*}
I_2 = \int^\infty_0 \int^\infty_0 a(x)K(x,y)u(x,t)u(y,t)\,dxdy
= \int^\infty_0 \int^\infty_0 a(y)K(x,y)u(x,t)u(y,t)\,dxdy.
\end{align*}
Hence,
\begin{align*}
I_2 = \frac{1}{2}\int^\infty_0  \int^\infty_0 \left\{a(x)+a(y)\right\}K(x,y)u(x,t)u(y,t)\,dxdy.
\end{align*}
Combining the expressions,
\begin{align*}
\frac{d}{dt}M^a(t) =I_1-I_2=\frac{1}{2}\int^\infty_0\int^\infty_0\{a(x+y)-a(x)-a(y)\}K(x,y)u(x,t)u(y,t)\,dxdy.
\end{align*}
\end{proof}

Next, we refine this identity using the truncated generalized moments.
\begin{thm}[Truncated generalized moment identity] \label{TMI}
Let $u$ be a weak solution to the SCE on $I=[0,T]$.
Then, for any $b\in C^0(\Rb )$,
the following identity holds:
\begin{align*}
\partial_t m^b (r,t) = \frac{1}{2}\iint_{D(r)} \big(&b(x+y)-b(x)-b(y)\big)K(x,y)u(x,t)u(y,t)\,dxdy \\
&-\int^r_0 b(x)u(x,t)\int^\infty_{r-x} K(x,y)u(y,t)\,dydx\quad(r\in \Rb,~\text{a.e.}~t\in  I)
\end{align*}
where $D(r) \coloneqq \{(x,y)\mid0 \leq x \leq r, 0\leq y \leq r-x\}$.
\end{thm}
\begin{proof}
From Proposition~\ref{weaksolpr}, 
for $b\in C^0(\Rb )$, $r\in\Rb$, and a.e. $t\in I$, 
\begin{align*}
\partial_t m^b(r,t)
=\int_{0}^{r} b(x)Q[u](x,t)\,dx 
=I_1-I_2,
\end{align*}
where
\begin{align*}
I_j:=\int_{0}^{r} b(x)Q_j[u](x,t)\,dx
\quad (j=1,2).
\end{align*}
We evaluate $I_1$:
  \begin{align*}
    I_1&=\frac{1}{2}\int_{0}^{r} b(x)\int_{0}^{x}K(x-y,y)u(y,t)u(x-y,t)\,dydx\\
    &=\frac{1}{2}\int^r_0 \int^r_y b(x)K(x-y,y)u(x-y,t)u(y,t)\,dxdy\quad\mbox{(Fubini)}\\
    &=\frac{1}{2}\int^r_0 \int^{r-y}_0 b(z+y)K(z,y)u(z,t)u(y,t)\,dzdy \quad\mbox{(change variable from $x$ to $z=x-y$)}\\
    &=\frac{1}{2}\iint_{D(r)} b(x+y)K(x,y)u(x,t)u(y,t)\,dxdy.\quad\mbox{(rename $z\to x$)}
  \end{align*}
  For $I_2$, define $E(r)\coloneqq \{(x,y)\mid0 \leq x \leq r, r-x< y\}$. Since $[0,r]\times \Rb$ is represented by the disjoint union of $D(r)$ and $E(r)$, we have
\begin{align*}
I_2 &=\int_{0}^{r}b(x)\int_{0}^{\infty}K(x,y)u(x,t)u(y,t)\,dydx\\
&=\iint_{D(r)\cup E(r)}  b(x)K(x,y)u(x,t)u(y,t)\,dydx \\
    &=\iint_{D(r)} \frac{b(x)+b(y)}{2}K(x,y)u(x,t)u(y,t)\,dxdy +\iint_{E(r)} b(x)K(x,y)u(x,t)u(y,t)\,dxdy,
  \end{align*}
where we used the identity:
\begin{align*}
    \iint_{D(r)}  b(x)K(x,y)u(x,t)u(y,t)\,dydx 
   = \iint_{D(r)}  b(y)K(x,y)u(x,t)u(y,t)\,dydx,
\end{align*}
which is derived from the symmetries of the region 
$D(r)$ and the coagulation kernel $K(x,y)$ with respect to $x$ and $y$.
Hence, we obtain
\begin{align*}
    \partial_t m^b(r,t) &=I_1 - I_2\\
    &=\frac{1}{2}\iint_{D(r)} \big(b(x+y)-b(x)-b(y)\big)K(x,y)u(x,t)u(y,t)\,dxdy \\
    &\quad-\int^r_0 b(x)u(x,t)\int^\infty_{r-x} K(x,y)u(y,t)\,dydx,
\end{align*}
for $r\in\Rb$ and a.e. $t\in I$.
\end{proof}

\begin{cor} \label{truncated moment identity cor}
  Let $u$ be a weak solution to the SCE on $I=[0,T]$. If $b\in C^0(\Rb;\Rb )$, then the following inequality holds:
  \begin{align*}
    \partial_t m^b(r,t)\leq\frac{1}{2}\iint_{D(r)}\big(b(x+y)-b(x)-b(y)\big)K(x,y)u(x,t)u(y,t)\,dxdy\quad
    (r\in \Rb,~\text{a.e.}\;t\in I).
  \end{align*}
Moreover, if $b$ is subadditive, then $\partial_t m^b(r,t)\leq 0$ holds for $r\in\Rb$ and a.e. $t\in I$.
\end{cor}
\begin{proof}
Since $b(x)\ge 0$, the first inequality follows directly from Theorem~\ref{TMI}. The second statement follows from the subadditivity condition: $b(x+y)-b(x)-b(y)\le 0$.
\end{proof}

\begin{prop}\label{Monotonically decreasing property of the moment}
Let $u$ be a weak solution to the SCE on $I=[0,T]$. 
We suppose that $b\in C^0(\Rb;\Rb)$ and $b$ is subadditive on $\Rb$.
If $M^b(0)<\infty$, then $M^b(t)$ is non-increasing in $t\in I$, that is, 
\begin{align}\label{MMM}
M^b(t)\le M^b(s)\le M^b(0)<\infty\quad (0\le s\le t\le T).
\end{align}
In particular, for $k=0$ or $k=1$, the moment $M_k(t)$ is non-increasing if $M_k(0)<\infty$. 
\end{prop}
\begin{proof}
By Corollary~\ref{truncated moment identity cor}, for $r\Rp$ and $t\in [0,T]$, we have
\[
m^b(r,t)\le m^b(r,s)\le m^b(r,0) \le M^b(0)<\infty.
\]
Taking the limit as $r\to \infty$ yields \eqref{MMM}. The final statement holds since $b(x)=1$ and $b(x)=x$ are subadditive.
\end{proof}

The following theorem describes the asymptotic behavior of the number of clusters $M_0(t)$ as $t\to\infty$,
and will be used in Section~\ref{sec:6}.
\begin{thm}\label{M0inf}
Let $u$ be a weak solution to the SCE on $\Rb$.
If $K(x,y)$ satisfies the following condition:
\begin{align}\label{Kr}
\inf_{x,y\geq r}K(x,y) >0\quad (r>0),
\end{align}
then it follows that
\begin{align}\label{limM0}
\lim_{t\to\infty}M_0(t)=0.
\end{align}
\end{thm}
\begin{proof}
We proceed by contradiction. Suppose that \eqref{limM0} is not true.
Then, since $M_0(t)$ is nonnegative and non-increasing, there exists $\varepsilon >0$ such that
\begin{align*}
\lim_{t\to\infty}M_0(t)=2\varepsilon.
\end{align*}
There also exists $r_0>0$ such that
\begin{align*}
m_0(r_0,0) =\int^{r_0}_0 u(x,0)\,dx \leq\varepsilon.
\end{align*}
Let $k_0\coloneq \inf_{x,y\geq r_0}K(x,y)>0$.
Then for any $s\geq0$, using Corollary~\ref{truncated moment identity cor} and Proposition~\ref{Monotonically decreasing property of the moment}, we have
\begin{align*}
M_0(s)\geq\lim_{t\to\infty}M_0(t)=2\varepsilon>\varepsilon\geq m_0(r_0,0)\geq m_0(r_0,s)\geq0.
  \end{align*}
Thus,
  \begin{align}
    M_0(s)-m_0(r_0,s)\geq \varepsilon \quad (s\in \Rb).\label{eq:3}
  \end{align}
Now apply Theorem~\ref{TMI} with $b(x)=1$. For $r>r_0$, we obtain
  \begin{align*}
    -\partial_t m_0 (r,t) &=\frac{1}{2} \int^r_0\int^{r-x}_0 K(x,y)u(x,t)u(y,t)\,dydx \\
    &\quad+\int^r_0\int^\infty_{r-x} K(x,y)u(x,t) u(y,t)\,dydx \\
    &\geq\frac{1}{2} \int^r_0\int^{\infty}_0 K(x,y)u(x,t)u(y,t)\,dydx \\
    &\geq\frac{1}{2} \int^r_{r_0}\int^r_{r_0} K(x,y)u(x,t)u(y,t)\,dydx \\
    &\geq\frac{k_0}{2} \int^r_{r_0}\int^r_{r_0} u(x,t)u(y,t)\,dydx \\
    &=\frac{k_0}{2}(m_0(r,t)-m_0(r_0,t))^2.
  \end{align*}
  Integrating both sides from $0$ to $t$:
  \begin{align*}
    m_0 (r,0)-m_0 (r,t) \geq\frac{k_0}{2}\int_{0}^{t}(m_0(r,s)-m_0(r_0,s))^2\,ds.
  \end{align*}
  Taking the limit $r\to\infty$, we have 
  \begin{align*}
    M_0(0)-M_0(t)&\geq\frac{k_0}{2}\lim_{r\to\infty}\int_{0}^{t}(m_0(r,s)-m_0(r_0,s))^2\,ds \\
    &= \frac{k_0}{2}\int_{0}^{t}(M_0(s)-m_0(r_0,s))^2\,ds \quad\text{(Lebesgue's
    DCT)}\\
    &\geq \frac{k_0}{2}\int_{0}^{t}\varepsilon^2\,ds \quad \text{(by \eqref{eq:3})}\\
    &=\frac{k_0\varepsilon^2}{2}t.
  \end{align*}
  Hence, 
  \begin{align*}
    M_0(t)\leq M_0(0) - \frac{k_0\varepsilon^2}{2}t
    \quad (t\in \Rb).
  \end{align*}
This implies $M_0(t)<0$ for $t>2M_0(0)/(k_0\vep^2)$, a contradiction. Thus, we conclude \eqref{limM0}.
\end{proof}

\section{Criteria for mass conservation}\label{sec:5}
\setcounter{equation}{0}
In this section, we extend the conditions for mass conservation established in \cite{I-K-M2023} to weak solutions.

For $b\in C^0(\Rb;\Rb)$, we define
\begin{align*}
  K_1^b(x,y) &\coloneqq(b(x)+x+1)(b(y)+y+1), \\
  K_2^b(x,y) &\coloneqq(b(y)+y+1)(x+1)+(b(x)+x+1)(y+1).
\end{align*}
We impose the following assumptions on the coagulation kernel $K$:
\begin{itemize} 
\item[(A1)] 
${}^\exists C_1 >0 \; \mathrm{s.t.}\; (x+y)K(x,y)\leq C_1 K_1^b(x,y)\;(x,y\in\Rb)$.
\item[(A2)]
${}^\exists C_2 >0 \; \mathrm{s.t.}\; \big(b(x+y)-b(x)-b(y)\big)K(x,y)\leq C_2 K_2^b(x,y)\;(x,y\in\Rb)$.
\end{itemize}
\begin{lem} \label{Sahidul-Kimura-Lem1}
Let $u$ be a weak solution to the SCE on $I=[0,T]$. 
Suppose that there exists $b\in C^0(\Rb;\Rb)$ such that (A2) holds and $M^b(0)<\infty$.
Then, for all $t\in I$, we have 
\[
M^b(t)\le L(t),
\]
where
\[
\mu_0:=\|u_0\|_{L^1_1(\Rp)}=M_0(0)+M_1(0),\quad L(t):=(M^b(0)+\mu_0)
e^{C_2\mu_0t}\quad (t\in I).
\]
\end{lem}
\begin{proof}
From Corollary~\ref{truncated moment identity cor}, for $r>0$ and $t\in I$, using the symmetry of $D(r)$ in $x$ and $y$, we obtain
\begin{align*}
\partial_t m^b (r,t) 
&\le
\frac{1}{2}\iint_{D(r)} \big(b(x+y)-b(x)-b(y)\big)K(x,y)u(x,t)u(y,t)\,dxdy \\
&\le
\frac{C_2}{2}\iint_{D(r)}\left\{(b(y)+y+1)(x+1)+(b(x)+x+1)(y+1)\right\}u(x,t)u(y,t)\,dxdy \\
&= C_2\iint_{D(r)}(b(x)+x+1)(y+1)u(x,t)u(y,t)\,dxdy \\
&\le 
C_2\int^r_0\int^r_0(b(x)+x+1)(y+1)u(x,t)u(y,t)\,dxdy \\
&= C_2(m^b (r,t) +m_1 (r,t) +m_0 (r,t))(m_1(r,t)+m_0(r,t)) \\
&\leq C_2(m^b(r,t) +M_1(0)+M_0(0))(M_1(0)+M_0(0)).
\quad\text{(from Proposition~\ref{Monotonically decreasing property of the moment})}\\
&= C_2(m^b(r,t) +\mu_0)\mu_0.
\end{align*}
Setting $p \coloneqq C_2 \mu_0$ and $q\coloneqq C_2\mu_0^2$, we get 
\begin{align*}
\partial_t m^b (r,t) \leq p m^b(r,t) + q
\quad (r\in\Rb,~t\in I).
\end{align*}
Then,
\begin{align*}
\partial_t (e^{-pt}m^b(r,t)) = e^{-pt}(\partial_t m^b(r,t) - pm^b(r,t))\leq qe^{-pt}.
\end{align*}
Integrating from $0$ to $t$ yields
\begin{align*}
e^{-pt}m^b(r,t)-m^b(r,0) \leq \int^t_0 qe^{-ps}\,ds = \frac{q}{p}(1-e^{-pt}).
\end{align*}
Hence,
\begin{align*}
m^b(r,t)&\leq m^b(r,0)e^{pt} + \frac{q}{p}(e^{pt}-1) 
\leq M^b(0)e^{pt} + \mu_0e^{pt}=L(t).
\end{align*}
Takking the limit as $r\to \infty$, 
we obtain $M^b(t)\leq L(t)$ for all $t\in I$.
\end{proof}
\begin{lem} \label{Sahidul-Kimura-Lem2}
Let $u$ be a weak solution to the SCE on $I=[0,T]$. Suppose there exists $b\in C^0(\Rb;\Rb)$ such that $M^b(0)<\infty$, and assumptions (A1) and (A2) hold.
Then, $J[u](r,t)<\infty$ for all $r\in\Rb$ and $t\in I$, and the following limits hold:
\begin{align*}
\lim_{r\to\infty}J[u](r,t)=0,\quad
\lim_{r\to\infty}\int^t_0 J[u](r,s)\,ds=0
\quad(t\in I).
\end{align*}
\end{lem}
\begin{proof}
For $r\in\Rp$, define 
\begin{align*}
E_0(r) \coloneqq 
\{(x,y)\mid 0\leq x\leq r, r-x\leq y\},~~
E_1(r) \coloneqq 
\Rb \times [r/2, \infty),~~
E_1'(r) \coloneqq 
[r/2,\infty)\times \Rb.
\end{align*}
Since $E_0(r)\subset (E_1(r)\cup E_1'(r))$ and by assumption (A1), 
we estimate:
$J[u](r,t)$ for $t\in I$ is evaluated as follows:
\begin{align*}
J[u](r,t) 
&= 
\iint_{E_0(r)} xK(x,y)u(x,t)u(y,t)\,dxdy \\
&\leq 
\iint_{E_1(r)\cup E_1'(r)} (x+y)K(x,y)u(x,t)u(y,t)\,dxdy \\
&\leq
2 \iint_{E_1(r)} (x+y)K(x,y)u(x,t)u(y,t)\,dxdy \\
&\leq
2C_1 \int^\infty_{r/2} \int_0^\infty(b(x)+x+1)(b(y)+y+1)u(x,t)u(y,t)\,dxdy \\
&= 
2C_1 \left(M^b(t)+M_0(t)+M_1(t)\right)\int^\infty_{r/2} (b(y)+y+1)u(y,t)\, dy \\
&\le 2C_1 \left(L(t)+\mu_0\right)\int^\infty_{r/2} (b(y)+y+1)u(y,t)\, dy,
  \end{align*}
where $L(t)$ and $\mu_0$ are defined in Lemma~\ref{Sahidul-Kimura-Lem1}.
From this, obviously, $J[u](r,t)\le 2C_1(L(t)+\mu_0)^2<\infty$ and  
$\lim_{r\to \infty}J[u](r,t)=0$ follow.
Applying Lebesgue's dominated convergence theorem,
we get
\begin{align*}
    \lim_{r\to\infty} \int^t_0 J[u](r,s)\,ds = 0 \quad(t\in I).
\end{align*}
\end{proof}

\begin{thm}\label{mainth1}
Let $u$ be a weak solution to the SCE on $I=[0,T]$. Suppose that there exists $b\in C^0(\Rb;\Rb)$ such that $M^b(0)<\infty$, and assumptions (A1) and (A2) hold.
Then, for all $t\in I$, 
\[
M_1(t)=M_1(0),\quad M^b(t)<\infty.
\]
\end{thm}
\begin{proof}
From Proposition~\ref{weaksolprop}~(3), we have
\begin{align*}
 m_1(r,t)=m_1(r,0)-\int^t_0 J[u](r,t)\,ds
 \quad (r\in\Rb,~t\in I).
\end{align*}
Taking the limit as $r\to\infty$ and applying Lemma~\ref{Sahidul-Kimura-Lem2}, we obtain
  \begin{align*}
    M_1(t) =M_1(0)\quad(t\in I).
  \end{align*}
The finiteness of $M^b(t)$ follows directly from Lemma~\ref{Sahidul-Kimura-Lem1}.
\end{proof}

The above theorem shows that mass conservation holds under conditions (A1) and (A2). The following examples of coagulation kernels $K(x,y)$ satisfying these assumptions were presented in \cite{I-K-M2023}.
\begin{prop}[\cite{I-K-M2023}] Let $b\in C^0(\Rb;\Rb)$ and $c\in\Rb$ satisfy
\begin{align*}
b(x+y)-b(x)-b(y)+x+y+c\ge 0\quad (x,y\in\Rb),
\end{align*}
and define the coagulation kernel as
\begin{align*}
    K(x,y):=\frac{b(x)+b(y)}{b(x+y)-b(x)-b(y)+x+y+c}.
\end{align*}
If this $K(x,y)$ is continuous on $\Rb\times \Rb$,
then, it satisfies assumptions (A1) and (A2). 
\end{prop}
In \cite{I-K-M2023}, the authors also provided the following examples of mass-conserving coagulation kernels and corresponding weight functions. These functions are not sublinear but they satisfy the conditions of Theorem~\ref{mainth1}:
\begin{align*}
&b(x)=x^3,~K(x,y)=\frac{x^2-xy+y^2}{xy+1}~\left(K(x,0)=x^2~\text{:~quadratic growth}\right),\\
&b(x)=e^x,~K(x,y)=\frac{e^x+e^y}{(e^x-1)(e^y-1)+x+y+1}~\left(K(x,0)=\frac{e^x+1}{x+1}~\text{:~exponential growth}\right).
\end{align*}

As mentioned in the introduction, if $\alpha,\beta\in[0,1]$ and $\alpha+\beta \le 1$, then the coagulation kernel $K(x, y) = x^\alpha y^\beta+x^\beta y^\alpha$ is mass-conserving (see \cite{E-L-M-P2003}). Applying Theorem~\ref{mainth1} to such a kernel leads to the following result.
\begin{prop}\label{hk-prop}
Suppose $0\leq\alpha\leq\beta\leq1$ and $\alpha+\beta\leq1$. Then the coagulation kernel $K(x,y)=x^\alpha y^\beta +x^\beta y^\alpha$ satisfies the assumptions (A1) and (A2) with  $b(x)\coloneqq x^{\beta+1}$. In particular, any weak solution to the SCE is mass-conserving if $M^b(0)<\infty$.
\end{prop}
\begin{proof}
For (A1), observe that
\begin{align*}
K_1^b(x,y) 
    &= (x^{\beta+1}+x+1)(y^{\beta+1}+y+1)\\
    &=x^{\beta+1}y^{\beta+1} + x^{\beta+1}y + xy^{\beta+1}  + x^{\beta+1}  + y^{\beta+1} + xy + x  + y + 1.
  \end{align*}
Using the inequality $x^p<x^q +1$ for $0\leq p\leq q$ and $x\in\Rb$, we estimate:
\begin{align*}
(x+y)K(x,y)&=x^{\alpha+1}y^\beta+x^{\beta+1}y^\alpha+x^\alpha y^{\beta+1}+x^\beta y^{\alpha+1}\\
    &< x^{\alpha+1}(y+1)+x^{\beta+1}(y+1)+(x+1)y^{\beta+1}+(x+1)y^{\alpha+1} \;(\because\beta\leq 1) \\
    &< (x^{\beta+1}+1)(y+1)+x^{\beta+1}(y+1)+(x+1)y^{\beta+1}+(x+1)(y^{\beta+1}+1) \\
    &=  2 x y^{\beta + 1}+ 2 x^{\beta + 1}y+2 x^{\beta + 1} + 2 y^{\beta + 1} + x + y + 2 \\
    &\leq 2 K_1^b(x,y).
  \end{align*}

For (A2), note that
 \begin{align*}
K_2^b(x,y)
=(y^{\beta+1}+y+1)(x+1)+(x^{\beta+1}+x+1)(y+1) > y^{\beta+1}(x+1).
\end{align*}
The case $x=0$ or $y=0$ is easy to verify, so we assume $0<x\le y$. 
Let $r:=y/x\ge 1$, then:
    \begin{align*}
    ({b(x+y)-b(x)-b(y)}) K(x,y)
    &=((x+y)^{\beta+1}-x^{\beta+1}-y^{\beta+1})(x^\alpha y^\beta + x^\beta y^\alpha)\\
    &=x^{\alpha+2\beta+1}\left( \left(r+1\right)^{\beta+1}-1-r^{\beta+1}\right)\left(r^{\alpha}+r^{\beta}\right)\\
    &=x^{\alpha+2\beta+1}\left(f(r)-1\right)(r^\alpha + r^\beta) \\
    &\leq 2x^{\alpha+2\beta+1}f(r)r,
  \end{align*}
where
$f(r)\coloneq (r+1)^{\beta+1} - r^{\beta+1}$.
Using the estimate:
\begin{align*}
  f(r) = \int_{r}^{r+1}(\beta +1)t^\beta\,dt 
  \leq (\beta +1)(r+1)^\beta 
  =(\beta +1)\left(\frac{r+1}{r}\right)^\beta r^\beta 
  \leq(\beta +1)2^\beta r^\beta,
\end{align*}
and applying $x^{\alpha+\beta}\leq 1+x$, we get:
\begin{align*}
  ({b(x+y)-b(x)-b(y)}) K(x,y)
  &\leq 
  (\beta +1)2^{\beta+1} x^{\alpha+2\beta+1}r^{\beta+1} \\
  &\leq(\beta +1)2^{\beta+1} (1+x)x^{\beta+1}r^{\beta+1} \\
  &=(\beta +1)2^{\beta+1}(1+x)y^{\beta+1} \\
  &< (\beta +1)2^{\beta+1} K_2^b(x,y).
\end{align*}
\end{proof}

\section{Criteria for gelation}\label{sec:6}
\setcounter{equation}{0}

The aim of this section is to derive sufficient conditions for gelation to occur using the generalized moment method.
We will make the following assumption on the weight function $b$ and the coagulation kernel $K$:
\begin{itemize} 
\item[(A3)] 
${}^\exists \lambda >0,\;{}^\exists \mu\ge 0\; \mathrm{s.t.}\; \big(b(x+y)-b(x)-b(y)\big)K(x,y)\le -\lambda xy+\mu (x+y+1)\;(x,y\in\Rb)$
\end{itemize}
\begin{thm}\label{mainth}
Suppose that the coagulation kernel $K(x,y)$ satisfies the condition in \eqref{Kr}, and that the weight function $b\in C^0(\Rb;\Rb)$ is non-decreasing on $\Rb$. If condition (A3) is satisfied, then the weak solution $u(x,y)$ of the SCE on $\Rb$ with $u_0\not\equiv 0$ and $M^b(0)<\infty$ undergoes gelation. That is, there exists $T>0$ such that $M_1(t)<M_1(0)$ for $t>T$. 
\end{thm}
\begin{proof}
we begin by noting that $b(x)\ge 0$ and $b(y)\le b(x+y)$,
which implies the inequality:
\begin{align*}
b(x)\ge \frac{1}{2}b(x)\ge \frac{1}{2}(b(x)+b(y)-b(x+y)).
\end{align*}
Applying this in Theorem~\ref{TMI}, we have
\begin{align*}
\partial_t m^b (r,t) &=\frac{1}{2} \int^r_0\int^{r-x}_0 
(b(x+y)-b(x)-b(y))K(x,y)u(x,t)u(y,t)\,dydx \\
&\quad -\int^r_0\int^\infty_{r-x} b(x)K(x,y)u(x,t) u(y,t)\,dydx \\
&\leq \frac{1}{2}\int^r_0\int^{\infty}_0 
(b(x+y)-b(x)-b(y))K(x,y)u(x,t)u(y,t)\,dydx \\
&\leq \frac{1}{2}\int^r_0\int^{\infty}_0 
(-\lambda xy+\mu (x+y+1))u(x,t)u(y,t)\,dydx \\
&=-\frac{\lambda}{2}m_1(r,t)M_1(t)
+\frac{\mu}{2}\left(
m_1(r,t)M_0(t)+m_0(r,t)M_1(t)
+m_0(r,t)M_0(t)\right)\\
&\le 
-\frac{\lambda}{2}m_1(r,t)M_1(t)
+\frac{\mu}{2}M_0(t)(2M_1(t)+M_0(t))\\
&\le 
-\frac{\lambda}{2}m_1(r,t)M_1(t)
+N(t),
\end{align*}
where we define
\begin{align*}
N(t):=\frac{\mu}{2}M_0(t)(2M_1(0)+M_0(0))
\end{align*}
Integrating both sides from $0$ to $t$,
we obtain:
\begin{align}\label{mbtN}
m^b (r,t)-m^b (r,0)
&\le
-\frac{\lambda}{2}\int_{0}^{t}
m_1(r,s)M_1(s)\,ds
+\int_0^tN(s)\,ds.
\end{align}
Since 
\begin{align*}
m^b (r,t)\le m^b (r,0) +\int_0^tN(s)\,ds\le M^b(0)+\int_0^tN(s)\,ds,
\end{align*}
it follows that $M^b(t)< M^b(0)+\int_0^tN(s)\,ds<\infty$.

Now, taking the limit as $r\to\infty$ in \eqref{mbtN} and applying Lebesgue's dominated convergence theorem, we obtain
\begin{align}\label{MMN}
M^b(t)-M^b(0)&\leq
-\frac{\lambda}{2}\int_{0}^{t}
M_1(s)^2\,ds
+\int_0^tN(s)\,ds.
\end{align}

To prove that gelation occurs, we proceed by contradiction. 
Assume that gelation does not occur, i.e., mass is conserved such that $M_1(t)=M_1(0)>0$ for all $t\in\Rb$.
Since $N(t)\to 0$ as $t\to\infty$ by Theorem~\ref{M0inf}, there exists a time $t_*>0$ such that
$N(t)\le \frac{\lambda}{4}M_1(0)^2$ for $t\ge t_*$.
Substituting this into \eqref{MMN} for $t>t_*$, we get:
\begin{align}
M^b(t)&\le M^b(0)-\frac{\lambda}{2}M_1(0)^2t
+\int_0^{t_*}N(s)\,ds+\int_{t_*}^tN(s)\,ds\label{fl}\\
&\le
M^b(0)-\frac{\lambda}{2}M_1(0)^2t
+\int_0^{t_*}N(s)\,ds
+\frac{\lambda}{4}M_1(0)^2(t-t_*)\notag\\
&\le
M^b(0)
+\int_0^{t_*}N(s)\,ds
-\frac{\lambda}{4}M_1(0)^2t\notag
\end{align}
For sufficiently large $t$, specifically when
\begin{align*}
t> \left(M^b(0)
+\int_0^{t_*}N(s)\,ds\right)\frac{4}{\lambda M_1(0)^2}, 
\end{align*}
it follows that $M^b(t)<0$,
which is a contradiction. Thus, mass conservation
does not hold and gelation occurs.
\end{proof}

\begin{thm}\label{mainth-2}
Assume that $b\in C^0(\Rb;\Rb)$ is non-decreasing on $\Rb$, and that assumption (A3) holds with $\mu=0$. Suppose further that $u$ is a weak solution to the SCE on $I=[0,T]$, with $u_0\not\equiv 0$ and $M^b(0)<\infty$.
If
\[
T > \frac{2 M^b(0)}{\lambda M_1(0)^2},
\]
then gelation occurs; that is, there exists $T_*\in (0,T)$ such that $M_1(t)<M_1(0)$ for $t\in (T_*,T]$. 
\end{thm}
\begin{proof}
In the proof of Theorem~\ref{mainth}, we may set $N(t)=0$. Therefore, inequality \eqref{fl} reduces to
\begin{align*}
M^b(t)\le M^b(0)-\frac{\lambda}{2}M_1(0)^2t.
\end{align*}
This leads a contradiction if $t>\frac{2 M^b(0)}{\lambda M_1(0)^2}$.
\end{proof}

\begin{prop}\label{prop-gel}
Suppose that $b\in C^0(\Rb; \Rb)$ is non-decreasing
and strictly subadditive on $\Rp$. Then, for $\vep >0$, $\lambda >0$, and $\mu\ge 0$, the coagulation kernel $K$ defined by
\begin{align*}
K(x,y)=\max\left(\frac{\lambda xy-\mu (x+y+1)}{b(x)+b(y)-b(x+y)},\varepsilon\right),
\end{align*}
satisfies the conditions (A3) and \eqref{Kr}. Therefore, the gelation occurs for any weak solution to the SCE on $\Rb$.
\end{prop}
\begin{proof}
Since $K(x,y)\ge \vep$, \eqref{Kr} is trivially satisfied.
From the strict subadditivity of $b$ on $\Rp$, we know that
$b(x)+b(y)-b(x+y)>0$ for $x,~y\in \Rp$.
Therefore, for $x,~y\in \Rp$. Thus, for $x,~y\in \Rp$, we obtain
\begin{align*}
(b(x)+b(y)-b(x+y))K(x,y)
&=\max\big(
\lambda xy-\mu (x+y+1),\;
\varepsilon (b(x)+b(y)-b(x+y))\big),\\
&\ge \lambda xy-\mu (x+y+1),
\end{align*}
which implies condition (A3). 
\end{proof}

\begin{ex}\label{ex1}
{\rm From Proposition~\ref{additive}, if $b\in C^0(\Rb; \Rb)$ is non-decreasing and strictly concave on $\Rp$, then 
$b$ is non-decreasing
and strictly subadditive on $\Rp$.
One such example is $b(x)=x^\alpha$ with $0\leq\alpha<1$. 
This coagulation kernel exhibits the gelation phenomenon described in Theorem~\ref{mainth}. However, for $k\in\Rp$,
along the line $y=kx$, 
we have:
\begin{align*}
K(x,kx)=\frac{\lambda kx^2-\mu ((k+1)x+1)}
{(1+k^\alpha -(1+k)^\alpha)x^\alpha}=
\frac{\lambda k}
{1+k^\alpha -(1+k)^\alpha}x^{2-\alpha}+
O(x^{1-\alpha})\quad\text{as}~x\to\infty.
\end{align*}
}
\end{ex}

\begin{ex}\label{ex2}
{\rm Another example of such coagulation kernels is
given by choosing 
\begin{align*}
b(x)=\frac{x}{\log(x+2)},
\end{align*}
in Proposition~\ref{prop-gel}.
Again, since it is non-decreasing and strictly concave on $\Rp$, we can conclude that
it is non-decreasing
and strictly subadditive on $\Rp$.

In fact, we can easily check that $b$ is increasing as follows:
\begin{align*}
b'(x)=\frac{(x+2)\log(x+2)-x}{(x+2)(\log(x+2))^2}>0 \quad (x\in\Rb).
\end{align*}
Similarly, $b''(x)<0$ can be shown, but since it is lengthy, we will not show it here, but will directly show the subadditivity of $b$ as follows:
\begin{align*}
b(x)+b(y)-b(x+y)
&=
\frac{x}{\log(x+2)}+\frac{y}{\log(y+2)}
-\frac{x+y}{\log(x+y+2)}
\ge  \frac{x+y-(x+y)}{\log(x+y+2)}= 0
\end{align*}
This coagulation kernel exhibits the gelation phenomenon from Theorem~\ref{mainth}. 

Let us check the asymptotic behavior of
$K(x,y)$ 
along the line $y=kx$, for $k\in\Rp$.
First, we note that
\begin{align}\label{ae1}
b(ax)=\frac{ax}{\log (ax+2)}
=
\frac{ax}{\log x}
-\frac{(a\log a)x}{(\log x)^2}
+O\left( \frac{x}{(\log x)^3}\right)
\quad \text{as}~x\to \infty\quad (a\in \Rp).
\end{align}
Using the asymptotic expansion \eqref{ae1}, for $k\in\Rp$, we obtain
\begin{align*}
b(x)+b(kx)-b(x+kx)
=\frac{H(k)x}{(\log x)^2}+
O\left( \frac{x}{(\log x)^3}\right)
\quad \text{as}~x\to \infty,
\end{align*}
where $H(k):=(k+1)\log(k+1)-k\log k >0$.
This implies that
\begin{align*}
\big(b(x)+b(kx)-b(x+kx)\big)^{-1}
=\frac{(\log x)^2}{H(k)x}+
O\left( \frac{\log x}{x}\right)
\quad \text{as}~x\to \infty.
\end{align*}
So, we conclude that
\begin{align*}
K(x,kx)=\frac{\lambda kx^2-\mu ((k+1)x+1)}
{b(x)+b(kx)-b(x+kx)}=
\frac{\lambda k}{H(k)}x(\log x)^2
+O(x\log x)\quad\text{as}~x\to\infty.
\end{align*}
In summary, this $K(x,y)$ has a growth order of only $O(x (\log x)^2)$ in the above sense, but is still an example of a coagulation kernel that allows a gelation phenomenon.

Figure~\ref{fig:4} presents numerical examples computed using the finite volume method proposed in \cite{Filbet-Laurencot2004}, for the case with parameters $\varepsilon = 0.001$, $T = 3$, and initial condition $u_0(x) = e^{-x}$. In this computation, the spatial domain $\Rb$ was truncated to the finite interval $[0, R]$ with sufficiently large $R > 0$, and a non-uniform mesh was generated by dividing $[0, R]$ into $500$ intervals using the transformation $\varphi(x) = R(x/R)^3$. The time step was set as $\Delta t = T/N$ (see \cite{Filbet-Laurencot2004,Miyata2025} for further details).

Figure~\ref{fig:4a} shows the time evolution of $u(x,t)$. Figures~\ref{fig:4b} to \ref{fig:4f} depict the evolution of $M_1(t)$ for the parameter sets $(R,N) = (50, 1000),(200, 1000),(1000, 1000),(10^5, 8\times 10^4),(1.5\times 10^5, 1.2\times 10^5)$, respectively. Numerical gelation is observed in all cases. Notably, for $R \ge 10^5$, the behavior of $M_1(t)$ becomes nearly independent of $R$, strongly suggesting that gelation occurs around $t \approx 2.5$ in this example.
}
\end{ex}

\begin{figure}[htb]
  \centering
  \begin{minipage}{0.43\columnwidth}
    \centering
    \includegraphics[width=\columnwidth]{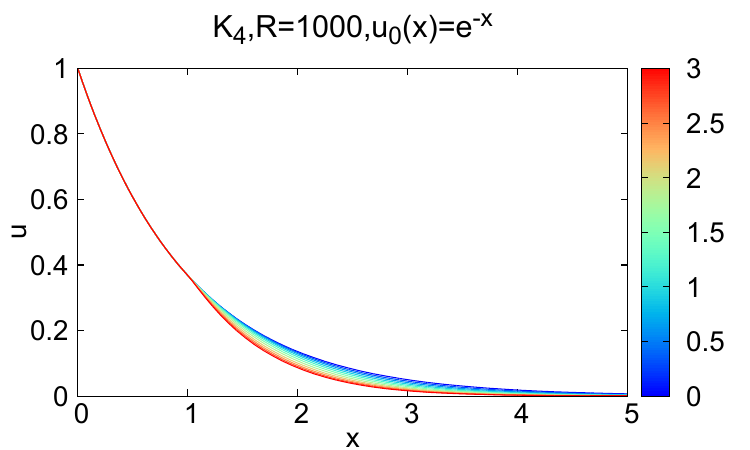}
    \subcaption{$u(x,t)$ $(x\in [0,5],~t\in [0,3])$}
    \label{fig:4a}
  \end{minipage}
  \hspace{5mm}
  \begin{minipage}{0.43\columnwidth}
    \centering
    \includegraphics[width=\columnwidth]{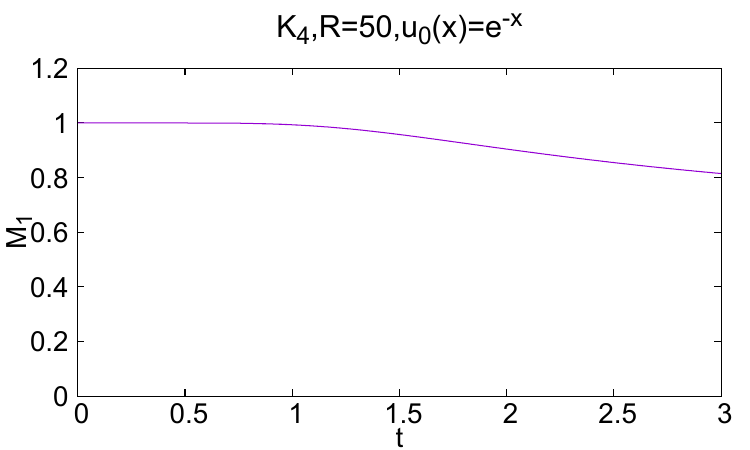}
    \subcaption{$M_1(t)$ $(t\in [0,3])$ with $R=50$}
    \label{fig:4b}
  \end{minipage}
  \\~\vspace{10pt}\\
  \begin{minipage}{0.43\columnwidth}
    \centering
    \includegraphics[width=\columnwidth]{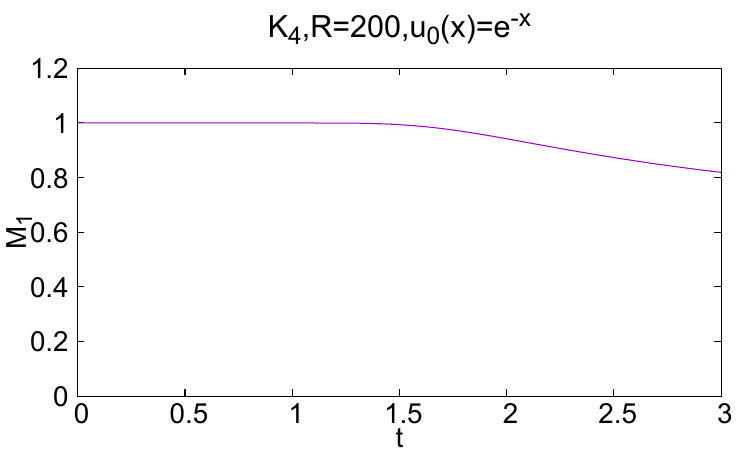}
    \subcaption{$M_1(t)$ $(t\in [0,3])$ with $R=200$}
    \label{fig:4c}
  \end{minipage}
  \hspace{5mm}
  \begin{minipage}{0.43\columnwidth}
    \centering
    \includegraphics[width=\columnwidth]{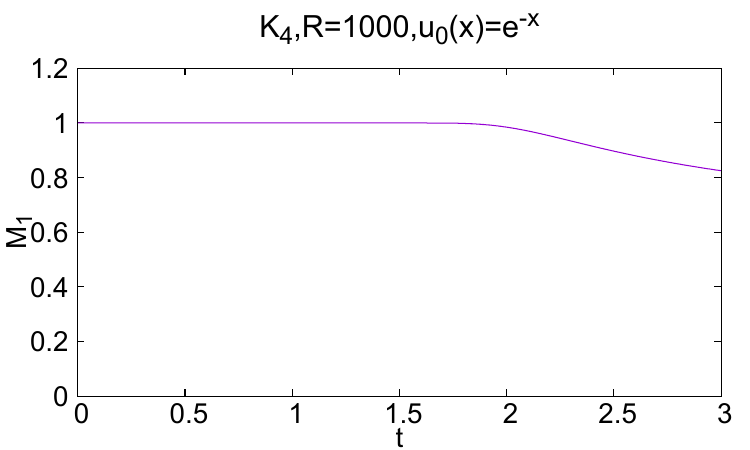}
    \subcaption{$M_1(t)$ $(t\in [0,3])$ with $R=1000$}
    \label{fig:4d}
  \end{minipage}
\\~\vspace{10pt}\\
  \begin{minipage}{0.43\columnwidth}
    \centering
    \includegraphics[width=\columnwidth]{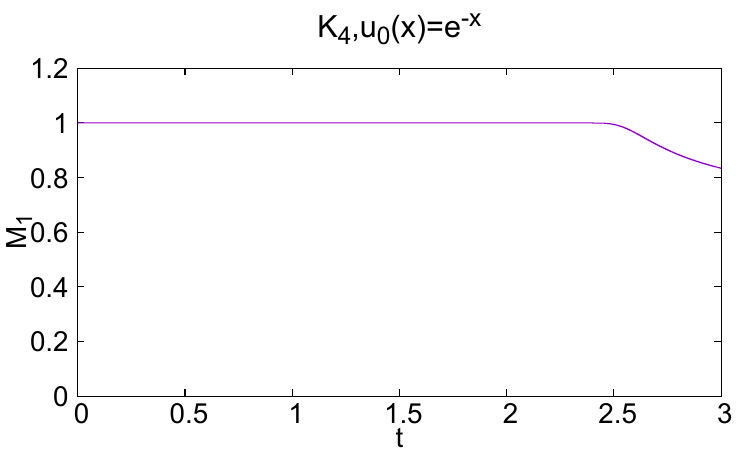}
    \subcaption{$M_1(t)$ $(t\in [0,3])$ with $R=10^5$}
    \label{fig:4e}
  \end{minipage}
  \hspace{5mm}
  \begin{minipage}{0.43\columnwidth}
    \centering
    \includegraphics[width=\columnwidth]{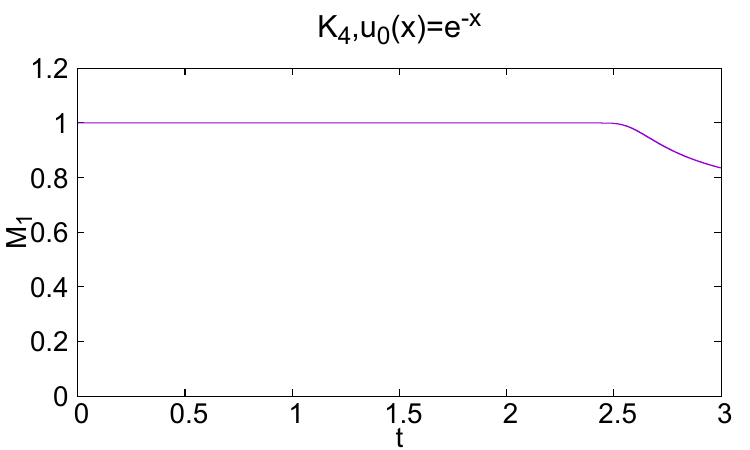}
    \subcaption{$M_1(t)$ $(t\in [0,3])$ with $R=1.5\times 10^5$}
    \label{fig:4f}
  \end{minipage}
  \caption{Numerical result for Example~\ref{ex2} with $u_0(x)=e^{-x}$}
  \label{fig:4}
\end{figure}

\section{Conclusion and discussion}\label{sec:7}
\setcounter{equation}{0}

In this paper, we investigated mass conservation and gelation phenomena for weak solutions to the Smoluchowski coagulation equation (SCE), employing the generalized moment method. Section~\ref{sec:2} introduced generalized moments and the notions of subadditivity and superadditivity for their associated weight functions. In Section~\ref{sec:3}, we defined weak solutions and the concept of mass flux, and explored their properties in detail. Section~\ref{sec:4} established foundational results concerning the generalized moments of weak solutions to the SCE, starting with the key result—the truncated generalized moment identity (Theorem~\ref{TMI}).

In Section~\ref{sec:5}, we presented one of the main results of this study: a sufficient condition for mass conservation (Theorem~\ref{mainth1}). This theorem generalizes earlier results obtained for classical solutions in \cite{I-K-M2023} to the setting of weak solutions. We demonstrated that a wide class of coagulation kernels \( K(x,y) \), including those with polynomial or even partial exponential growth, can still ensure mass conservation. As an application of this result, we recovered the known mass conservation property for homogeneous kernels (Proposition~\ref{hk-prop}).

The other main result is given in Section~\ref{sec:6}, where we established sufficient conditions under which gelation occurs (Theorems~\ref{mainth} and \ref{mainth-2}). We also provided explicit examples of coagulation kernels satisfying these conditions, including a family of kernels exhibiting growth of the order \( O(x(\log x)^2) \). These kernels are shown to induce gelation. The theoretical findings were further supported by detailed numerical experiments using a finite volume scheme.

The techniques developed in this work contribute to the mathematical theory of coagulation equations and offer a unified framework for analyzing both mass-conserving and gelation regimes. Future directions include extending the present analysis to systems involving fragmentation, diffusion, or spatial inhomogeneity. Another important avenue is the refinement of numerical methods to efficiently capture critical gelation dynamics and to accommodate more complex kernel structures. As demonstrated in this study, the synergy between rigorous analysis and computational experiments provides a powerful approach for advancing the understanding of coagulation phenomena.

~\\\noindent
{\bf Acknowledgment}:
This work was partially supported by JSPS KAKENHI Grant Nos. 25K00920, and 24H00184. 


\end{document}